\documentclass[12pt]{article}

\usepackage{amsmath}
\usepackage{amssymb}
\usepackage[mathscr]{eucal}
\usepackage{graphicx}
\usepackage{color}

\usepackage{bm}
\usepackage{mathrsfs}
\usepackage{amsthm}
\usepackage{enumerate}
\usepackage[mathscr]{eucal}
\usepackage{graphicx}

\newtheorem{Theorem}{\sc Theorem}
\newtheorem{Definition}[Theorem]{\sc Definition}
\newtheorem{Proposition}[Theorem]{\sc Proposition}
\newtheorem{Lemma}[Theorem]{\sc Lemma}

\newtheorem{Remark}[Theorem]{\sc Remark}
\newtheorem{Example}[Theorem]{\sc Example}

\newcommand{\R}{{\if mm {\rm I}\mkern -3mu{\rm R}\else \leavevmode
		\hbox{I}\kern -.17em\hbox{R} \fi}}

\def\sqr#1#2{{
		\vcenter{
			\vbox{\hrule height.#2pt
				\hbox{\vrule width.#2pt height#1pt \kern#1pt
					\vrule width.#2pt
				}
				\hrule height.#2pt
			}
		}
}}

\def\bar{\overline}
\def\real{\mathbb{R}}
\def\nat{\mathbb{N}}

\def\lista#1
{{ \itemindent 0.0cm \labelsep .2cm \leftmargin 0.8cm \rightmargin
		0.0cm \labelwidth 0.6cm \topsep 0.0mm
		\parsep 0.0mm
		\itemsep 0.0mm
		\begin{list}{}
			{ \setlength{\leftmargin}{.8cm} \setlength{\rightmargin}{0.0cm}
				\setlength{\parsep}{0.0mm} \setlength{\topsep}{.0mm}
				\setlength{\parskip}{.0cm} \setlength{\itemsep}{.0cm} }
			{#1}\end{list}} }

\begin{document}
\title{
Existence, comparison, and convergence results
for a class of elliptic hemivariational inequalities
\thanks{
\,  This project has received funding from the European Union's Horizon 2020
Research and Innovation Programme under the Marie Sk{\l}odowska-Curie grant agreement No. 823731 CONMECH.
It is supported by
NSF of Guangxi, Grant No: 2018GXNSFAA281353,
Beibu Gulf University Project No. 2018KYQD06, and the projects financed by the Ministry of Science and
Higher Education of Republic of Poland under
Grants Nos. 4004/GGPJII/H2020/2018/0 and
440328/PnH2/2019. The first and the fourth authors are also
partially sponsored by
the Project PIP No. 0275 from CONICET-UA, Rosario, Argentina.
}}

\author{
Claudia M. Gariboldi  \footnote{\, Depto. Matem\'atica, FCEFQyN, Univ. Nac. de R\'io Cuarto, Ruta 36 Km 601,
5800 R\'io Cuarto, Argentina. E-mail: cgariboldi@exa.unrc.edu.ar.}, \ \
Stanis{\l}aw Mig\'orski \footnote{\, College of Applied Mathematics, Chengdu University of Information Technology, Chengdu 610225, Sichuan Province, P.R. China, and Jagiellonian University in Krakow, Chair of Optimization and Control, ul. Lojasiewicza 6, 30348 Krakow, Poland.
E-mail address: stanislaw.migorski@uj.edu.pl.}, \ \
Anna Ochal \footnote{\, Jagiellonian University in Krakow, Chair of Optimization and Control, ul. Lojasiewicza 6, 30348 Krakow, Poland. E-mail address: anna.ochal@uj.edu.pl.}, \\
and \ Domingo A. Tarzia \footnote{\, Depto. Matem\'atica-CONICET, FCE, Univ. Austral, Paraguay 1950, S2000FZF Rosario, Argentina. E-mail: DTarzia@austral.edu.ar.}
}

\date{}

\maketitle
	
\noindent {\bf Abstract.} \
In this paper we study a class of elliptic boundary hemivariational inequa\-lities which originates in the steady-state heat conduction problem with nonmonotone multivalued subdifferential boundary condition on a portion of the boundary described by the Clarke generalized gradient
of a locally Lipschitz function.
First, we prove a new existence result for the inequality employing the theory of pseudomonotone operators.
Next, we give a result on comparison of solutions,
and provide sufficient conditions that guarantee the asymptotic behavior of solution, when the heat transfer coefficient tends to infinity.
Further, we show a result on the continuous dependence of solution on the internal energy and heat flux.
Finally, some examples of convex and nonconvex potentials
illustrate our hypotheses.



\medskip
	
	

\medskip
	

\medskip

{\thispagestyle{empty}} 


\section{Introduction}

We consider a bounded domain $\Omega$ in $\real^d$ whose
regular boundary $\Gamma $ consists of the union of three disjoint portions $\Gamma_{i}$, $i=1$, $2$, $3$
with $|\Gamma_{i}|>0$, where $|\Gamma_i|$ denotes the $(d-1)$-dimensional
Hausdorff measure of the portion $\Gamma_i$ on $\Gamma$.
The outward normal vector on
the boundary is denoted by $n$.
We formulate the following two steady-state heat conduction problems
with mixed boundary conditions:
\begin{eqnarray}
&&
-\Delta u=g \ \ \mbox{in} \ \ \Omega,
\ \ \quad u\big|_{\Gamma_{1}}=0,
\ \ \quad-\frac{\partial u}{\partial n}\big|_{\Gamma_{2}}=q,
\ \ \quad u\big|_{\Gamma_{3}}=b,
\label{P} \\[2mm]
&&
-\Delta u=g \ \ \mbox{in} \ \ \Omega,
\ \ \quad u\big|_{\Gamma _{1}}=0,
\ \ \quad
-\frac{\partial u}{\partial n}\big|_{\Gamma_{2}}=q,
\ \ \quad -\frac{\partial u}{\partial n}\big|_{\Gamma_{3}}
=\alpha (u-b),
\label{Palfa}
\end{eqnarray}
where $u$ is the temperature in $\Omega$,
$g$ is the internal energy in $\Omega$,
$b$ is the temperature on $\Gamma_{3}$ for (\ref{P})
and the temperature of the external
neighborhood of $\Gamma_{3}$ for (\ref{Palfa}),
$q$ is the heat flux on $\Gamma_{2}$
and $\alpha>0$ is the heat transfer coefficient on $\Gamma_{3}$,
which satisfy the hypothesis:
$g\in L^2(\Omega)$,
$q\in L^2(\Gamma_2)$ and
$b\in H^{\frac{1}{2}}(\Gamma_3)$.

Throughout the paper we use the following notation
\begin{eqnarray*}
&& V=H^{1}(\Omega),
\quad
V_{0}=\{v\in V \mid v = 0 \ \ \mbox{on} \ \ \Gamma_{1} \},
\\[2mm]
&&
K=\{v\in V \mid
v = 0 \ \ \mbox{on} \ \ \Gamma_{1},\
v = b \ \ \mbox{on} \ \ \Gamma_{3} \},
\quad
K_{0}=\{v\in V \mid
v = 0 \ \ \mbox{on} \ \ \Gamma_{1}\cup \Gamma_3 \}, \\[2mm]
&&
a(u,v)=\int_{\Omega }\nabla u \, \nabla v \, dx,
\quad a_{\alpha}(u,v)=a(u,v)+\alpha \int\limits_{\Gamma_3}\gamma(u) \gamma (v) d\Gamma, \\
&&
L(v)= \int_{\Omega}g v \,dx -
\int_{\Gamma_{2}}q \gamma (v) \,d\Gamma,
\quad
L_{{\alpha}}(v)=L(v)
+\alpha \int\limits_{\Gamma_3}b \gamma(v) \, d\Gamma,
\end{eqnarray*}
where $\gamma \colon V \to L^2(\Gamma)$
denotes the trace operator on $\Gamma$.
In what follows, we write $u$ for the trace of a function $u \in V$ on the boundary.
In a standard way, we obtain the following variational formulations of (\ref{P}) and (\ref{Palfa}), respectively:
\begin{eqnarray}
&&
\hspace{-1cm}
\mbox{find} \ \ u_{\infty}\in K \ \ \mbox{such that}\ \
a(u_{\infty},v)=L(v)
\ \ \mbox{for all} \ \ v\in K_{0},
\label{Pvariacional} \\[2mm]
&&
\hspace{-1cm}
\mbox{find} \ \ u_{\alpha}\in V_0 \ \ \mbox{such that}\ \
a_{\alpha}(u_{\alpha },v)=L_{\alpha }(v)
\ \ \mbox{for all} \ \ v\in V_{0}.
\label{Palfavariacional}
\end{eqnarray}
The standard norms on $V$ and $V_0$ are denoted by
\begin{eqnarray*}
&&
\| v \|_V = \Big( 
\| v \|^2_{L^2(\Omega)}
+ \| \nabla v \|^2_{L^2(\Omega;\real^d)} \Big)^{1/2}
\ \ \mbox{for} \ \ v \in V, \\ [2mm]
&&
\| v \|_{V_0} = \| \nabla v \|_{L^2(\Omega;\real^d)}
\ \ \mbox{for} \ \ v \in V_0.
\end{eqnarray*}
It is well known by the Poincar\'e inequality,
see~\cite[Proposition 2.94]{CLM}, that on $V_0$ the above two norms
are equivalent. Note that the form $a$ is bilinear, symmetric, continuous and coercive with constant $m_a > 0$, i.e.
\begin{equation}\label{coercive}
a(v, v) = \|v\|^{2}_{V_0} \ge m_a \|v\|^{2}_{V}
\ \ \mbox{for all} \ \ v\in V_{0}.
\end{equation}

It is well known that the regularity of solution to
the mixed elliptic problems (\ref{P}) and (\ref{Palfa}) is problematic in the neighborhood of a part of the boundary,
see for example the monograph~\cite{G}.
A regularity results for elliptic problems
with mixed boundary conditions can be found
in~\cite{AK,BBP,LCB}. Moreover, sufficient hypothesis on the
data in order to have $H^{2}$ regularity for elliptic variational inequalities is given in~\cite{R}.
We remark that, under additional hypotheses on the data
$g$, $q$ and $b$, problems (\ref{P}) and (\ref{Palfa})
can be considered as steady-state two phase Stefan problems, see, for example,~\cite{GT,TT,Ta1,Ta3}.

The problems (\ref{Pvariacional}) and (\ref{Palfavariacional})
have been extensively studied in several papers such
as~\cite{GT,TT,Ta,Ta1,Ta2}.
Some properties of monotonicity and convergence, when the parameter $\alpha$ goes to infinity,
obtained in the aforementioned works, are recalled in the following result.
\begin{Theorem}\label{teor1}
If the data satisfy
$b=const.> 0$, $g\in L^2(\Omega)$ and $q\in L^2(\Gamma_2)$ with the properties $q\geq 0$ on $\Gamma_{2}$ and
$g\leq 0$ in $\Omega$, then
\begin{itemize}
\item [\rm (i)] $u_{\infty}\leq b$ \ {\rm in} \ $\Omega$,
\item [\rm (ii)] $u_{\alpha}\leq b$ \ {\rm in} \ $\Omega$,
\item [\rm (iii)] $u_{\alpha}\leq u_{\infty}$
\ {\rm in} \ $\Omega$,
\item [\rm (iv)] {\rm if} \ $\alpha_{1} \leq \alpha_{2}$,
\ {\rm then} \ $u_{\alpha_{1}}\leq u_{\alpha_{2}}$ \ {\rm in} \ $\Omega$,
\item [\rm (v)] $u_{\alpha} \to u_{\infty}$ \ {\rm in} \ $V$,
\ {\rm as} \ $\alpha \to \infty$.
\end{itemize}
\end{Theorem}

The main goal of this paper is to study a generalization
of problem (\ref{Palfa}) for which we provide sufficient conditions that guarantee
the comparison properties and asymptotic behavior, as $\alpha \to \infty$, stated in Theorem~\ref{teor1}.
Moreover, for a more general problem, we also show a result on the continuous dependence of solution
on the data $g$ and $q$.

The mixed nonlinear boundary value problem for the elliptic equation under consideration reads as follows.
\begin{equation}\label{Pjalfa}
-\Delta u=g \ \ \mbox{in} \ \ \Omega,
\ \quad u\big|_{\Gamma _{1}}=0,
\ \quad  -\frac{\partial u}{\partial n}\big|_{\Gamma_{2}}=q,  \ \quad -\frac{\partial u}{\partial n}\big|_{\Gamma_{3}}
\in \alpha \, \partial j(u).
\end{equation}
Here $\alpha$ is a positive constant while
the function $j \colon \Gamma_{3} \times \real \to \real$, called a superpotential (nonconvex potential),
is such that $j(x, \cdot)$ locally Lipschitz for a.e. $x \in \Gamma_3$
and not necessary differentiable.
Since in general $j(x, \cdot)$ is nonconvex, so the multivalued condition on $\Gamma_3$ in problem (\ref{Pjalfa})
is described by a nonmonotone relation expressed by the generalized gradient of Clarke.
Such multivalued relation in problem (\ref{Pjalfa}) is met
in certain types of steady-state heat conduction problems
(the behavior of a semipermeable membrane of finite
thickness, a temperature control problems, etc.).
Further, problem (\ref{Pjalfa}) can be considered as a prototype of several boundary semipermeability models,
see~\cite{MO,NP,P,ZLM}, which are motivated by problems arising in hydraulics, fluid flow problems through porous media,
and electrostatics, where the solution represents the pressure and the electric potentials.
Note that the analogous problems with maximal monotone multivalued boundary relations (that is the case when $j(x, \cdot)$
is a convex function) were considered in~\cite{Barbu,DL},
see also references therein.

Under the above notation,
the weak formulation of the elliptic problem (\ref{Pjalfa})
becomes the following boundary hemivariational inequality:
\begin{equation}\label{Pj0alfavariacional}
\mbox{find} \ \ u \in V_0 \ \ \mbox{such that} \ \
a(u,v) + \alpha \int_{\Gamma_{3}}j^{0}(u;v)\, d\Gamma
\geq L(v) \ \ \mbox{\rm for all} \ \  v\in V_{0}.
\end{equation}
Here and in what follows we often omit the variable
$x$ and we simply write $j(r)$ instead of $j(x, r)$.
Observe that if $j(x, \cdot)$ is a convex function
for a.e. $x \in \Gamma_3$, then
the problem \eqref{Pj0alfavariacional} reduces
to the variational inequality of second kind:
\begin{equation}\label{VI11}
\mbox{find} \ u \in V_0 \ \mbox{such that} \ 	
a(u,v-u) +\alpha \int_{\Gamma_{3}}(j(v) - j(u)) \, d\Gamma
\geq L(v-u) \ \ \mbox{\rm for all} \ \  v\in V_{0}.
\end{equation}
Note that when $j(r) = \frac{1}{2} (r-b)^2$, problem (\ref{VI11}) reduces to a variational inequality corresponding to problem (\ref{Palfa}).
Several other examples of convex potentials can be found in various diffusion problems. For instance, the following convex functions:
$$ j(r) = |r|,\ \ \
j(r) =
\begin{cases}
\beta (r-c)^5 &\text{{\rm if} \ $r \ge c,$} \\
0 &\text{{\rm if} \ $r < c,$}
\end{cases}
\ \ \mbox{and} \ \
j(r) =
\begin{cases}
\beta r^{9/4} &\text{{\rm if} \ $r \ge 0,$} \\
0 &\text{{\rm if} \ $r < 0,$}
\end{cases}
$$
with suitable constants $\beta > 0$ and $c \in \real,$
appear in models which describe a free boundary problem with Tresca condition, see~\cite{BUTA},
the Stefan-Boltzman heat radiation law,
and the natural convection, respectively, see~\cite{Barbu,Kawohl}, and the references therein
for further applications and extensions.
On the other hand, the stationary heat conduction models with nonmonotone multivalued subdifferential interior and boundary semipermeability relations
can not be described by convex potentials.
They use locally Lipschitz potentials and their weak formulations lead to hemivariational inequalities, see~\cite[Chapter~5.5.3]{NP} and~\cite{P}.

We mention that theory of hemivariational and variational inequalities has been proposed in the 1980s
by Panagiotopoulos, see~\cite{NP,P0,P1},
as variational formulations of important classes of inequality problems in mechanics.
In the last few years, new kinds of variational, hemivariational, and variational-hemivariational
inequalities have been investigated, see recent monographs~\cite{CLM,MOS,SM},
and the theory has emerged today as a new and interesting branch of applied mathematics.

The rest of the paper is structured as follows.
In Section~\ref{Preliminaries} we provide a new existence
result for problem (\ref{Pj0alfavariacional}).
In Section~\ref{Comparison} we establish two comparison properties for solutions to problem (\ref{Pj0alfavariacional}).
The convergence result of solution of problem  (\ref{Pj0alfavariacional}) to the solution of problem  (\ref{Pvariacional}), when the parameter $\alpha$ goes to
infinity, is provided in Section~\ref{Asymptotic}.
In Section~\ref{Continuous} we study the continuous dependence
of solution to problem (\ref{Pj0alfavariacional}) on the internal energy $g$ and the heat flux $q$.
The proofs are based on arguments of compactness, lower semicontinuity, monotonicity, various estimates,
the theory of elliptic hemivariational
inequalities and nonsmooth analysis~\cite{C,DMP,DMP1,DL,GLMOP,MOS,
P1,SST,SM,Z}.
Finally, in Section~\ref{Examples}
we deliver several examples of convex and nonconvex potentials
which satisfy the hypotheses on function $j$ required
in this paper.

\section{Preliminaries}\label{Preliminaries}

In this section first recall standard notation and preliminary concepts, and then provide a new result on existence of solution to the elliptic hemivariational
inequality~(\ref{Pj0alfavariacional}).

Let $(X, \| \cdot \|_{X})$ be a reflexive Banach space,  $X^{*}$ be its dual, and $\langle \cdot, \cdot \rangle$ denote the duality between $X^*$ and $X$.
For a real valued function defined
on $X$, we have the following
definitions~\cite[Section~2.1]{C} and~\cite{DMP,MOS}.
\begin{Definition}
	A function $\varphi \colon X\rightarrow \mathbb{R}$
	is said to be locally Lipschitz, if for every $x\in X$
	there exist $U_{x}$ a neighborhood of $x$ and a constant $L_{x}>0$ such that
	$$
	|\varphi(y)-\varphi(z)|\leq L_{x}\|y-z\|_{X}
	\ \ \mbox{\rm for all} \ \ y, z\in U_{x}.
	$$
	For such a function the generalized (Clarke) directional derivative of $j$ at the point $x\in X$ in the direction
	$v\in X$ is defined by
	$$
	\varphi^{0}(x;v)=\limsup\limits_{y \rightarrow x, \, \lambda \rightarrow 0^{+}}
	\frac{\varphi(y +\lambda v)-\varphi(y)}{\lambda} \, .
	$$
	The generalized gradient (subdifferential)
	of $\varphi$ at $x$ is a subset of the dual space $X^{*}$ given by
	$$
	\partial \varphi(x)=\{\zeta\in X^{*} \mid \varphi^{0}(x;v)\geq \langle
	\zeta,v\rangle \ \ \mbox{\rm for all} \ \  v \in X\}.
	$$
\end{Definition}

We shall use the following properties of the generalized directional derivative and the generalized
gradient, see~\cite[Proposition~3.23]{MOS}.
\begin{Proposition}\label{PROP1}
	Assume that $\varphi\colon X\rightarrow \mathbb{R}$ is a locally Lipschitz function. Then the following hold:
	\begin{itemize}
		\item [\rm (i)]
		for every $x\in X$, the function $X\ni v\mapsto \varphi^{0}(x;v)\in \mathbb{R}$ is positively homogeneous,
		and subadditive, i.e.,
		\begin{eqnarray*}
			&&
			\varphi^{0}(x;\lambda v)=\lambda \varphi^{0}(x;v)
			\ \ \mbox{\rm for all} \ \ \lambda \geq 0, \ v\in X, \\[1mm]
			&&
			\varphi^0(x; v_1+v_2) \le
			\varphi^0(x; v_1) + \varphi^0(x; v_2)
			\ \ \mbox{\rm for all} \ \ v_1, v_2 \in X,
		\end{eqnarray*}
		respectively.
		\item [\rm (ii)]
		for every $x\in X$, we have $\varphi^{0}(x;v)
		=\max\{\langle\zeta, v\rangle \mid \zeta \in \partial \varphi(x)\}$.
		\item [\rm (iii)]
		the function $X \times X \ni (x, v) \mapsto \varphi^0(x; v) \in \real$ is upper semicontinuous.
        \item [\rm (iv)]
        for every $x\in X$, the gradient $\partial\varphi (x)$ is a nonempty, convex, and weakly compact subset of~$X^*$.
        \item [\rm (v)]
        the graph of the generalized gradient $\partial \varphi$ is closed in $X\times ({\it weak\mbox{--}}X^*)$--topology.
	\end{itemize}
\end{Proposition}

Now, we pass to a result on existence of solution to the elliptic hemivariational inequality:
\begin{equation}\label{abstract}
\mbox{\rm find} \ \ u \in V_0 \ \ \mbox{\rm such that} \ \
a(u, v) + \alpha \int_{\Gamma_{3}}j^{0}(u; v)\, d\Gamma
\geq f(v) \ \ \mbox{\rm for all} \ \  v\in V_{0}.
\end{equation}
We admit the following standing hypothesis.

\medskip

\noindent
${\underline{H(j)}}$: $j\colon \Gamma_3 \times \real \to \real$ is such that

\smallskip

\noindent
\quad (a) $j(\cdot, r)$ is measurable for all $r \in \real$,

\smallskip

\noindent
\quad (b)
$j(x, \cdot)$ is locally Lipschitz for a.e. $x \in \Gamma_3$,

\smallskip

\noindent
\quad (c)
there exist $c_0$, $c_1 \ge 0$ such that
$| \partial j(x, r)| \le c_0 + c_1 |r|$
for all $r \in \real$, a.e. $x\in \Gamma_3$,

\smallskip

\noindent
\quad (d)
$j^0(x, r; b-r) \le 0$ for all $r \in \real$,
a.e. $x \in \Gamma_{3}$ with a constant $b \in \real$.

\medskip


\medskip

Note that the existence results for elliptic hemivariational inequalities can be found in several contributions, see~\cite{CLM,unified,ELAS,MOS,NP}.
In comparison to other works, the new hypothesis is $H(j)$(d).
Under this condition we will show both existence of solution
to problem \eqref{abstract}
and a convergence result when $\alpha \to \infty$.
We underline that, if the hypothesis $H(j)$(d) is replaced
by the relaxed monotonicity condition
(see Remark~\ref{RRCC} for details)
\begin{equation*}
j^0(x, r; s-r) + j^0(x,s; r-s) \le m_j \, |r-s|^2
\end{equation*}
for all $r$, $s \in \real$, a.e. $x\in\Gamma_3$
with $m_j \ge 0$,
and the following smallness condition
$$
m_a > \alpha \, m_j \| \gamma\|^2
$$
is assumed, then
problem (\ref{abstract}) is uniquely solvable,
see~\cite[Lemma~20]{ELAS} for the proof.
However, this smallness condition is not suitable in the study of problem \eqref{abstract} since for a sufficiently large value of
$\alpha$, it is not satisfied.

In the following result we apply a surjectivity result in~\cite[Proposition~3.61]{MOS} and partially follow arguments of~\cite[Lemma~20]{ELAS}.
For completeness we provide the proof.

\smallskip

\begin{Theorem}\label{existence}
If $H(j)$ holds, $f \in V_0^*$ and $\alpha > 0$,
then the hemivariational inequality \eqref{abstract}
has a solution.
\end{Theorem}
\begin{proof}
Let $\langle \cdot, \cdot \rangle$ stand for the duality pairing between $V^*_0$ and $V_0$.
Let $A \colon V_0 \to V_0^*$ be defined by
$$
\langle Au, v \rangle = a(u, v)
\ \ \mbox{for} \ \ u, v \in V_0.
$$
It is obvious that the operator $A$ is linear, bounded and coercive, i.e.,
$\langle Av, v \rangle \ge \| v \|^2_{V_0}$
for all $v \in V_0$.
Moreover, let
$J \colon L^2(\Gamma_3) \to \real$ be given
by
\begin{equation*}
J(w) = \int_{\Gamma_3} j(x, w(x)) \, d\Gamma
\ \ \mbox{for all} \ \ w \in L^2(\Gamma_3).
\end{equation*}

\noindent
From $H(j)$(a)-(c), by~\cite[Corollary~4.15]{MOS}, we infer that the functional $J$ enjoys the following pro\-per\-ties:
\begin{itemize}
\item[(p1)]
$J$ is well defined and Lipschitz continuous on bounded subsets of $L^2(\Gamma_3)$, hence also locally Lipschitz,
\item[(p2)]
$\displaystyle
J^0(w; z) \le \int_{\Gamma_3} j^0(x, w(x); z(x))\, d\Gamma$
for all $w$, $z \in L^2(\Gamma_3)$,
\item[(p3)]
$\| \partial J(w) \|_{L^2(\Gamma_3)} \le
{\bar{c}}_0 + {\bar{c}}_1 \, \| w \|_{L^2(\Gamma_3)}$
for all $w \in L^2(\Gamma_3)$ with ${\bar{c}}_0$,  ${\bar{c}}_1 \ge 0$.
\end{itemize}

We introduce the operator $B \colon V_0 \to 2^{V_0^*}$
defined
by
$$
Bv = \alpha\, \gamma^* \partial J (\gamma v)
\ \ \mbox{for all} \ \ v \in V_0,
$$
where $\gamma^* \colon L^2(\Gamma) \to V_0^*$
denotes the adjoint to the trace $\gamma$.

We show that $B$ is pseudomonotone and bounded
from $V_0$ to $2^{V_0^*}$, see~\cite[Definition~3.57]{MOS}.
By~Proposition~\ref{PROP1}\,(iv), it follows that
the values of $\partial J$ are nonempty, convex and weakly compact subsets of $L^2(\Gamma_3)$.
Hence, the set $Bv$ is nonempty, closed and convex in $V_0^*$ for all $v \in V_0$.
The operator $B$ is bounded which is a consequence of the following estimate
\begin{equation*}\label{ESTB}
\| Bv\|_{V_0^*}
\le \alpha \, \| \gamma^*\| \, \|
\partial J (\gamma v) \|_{L^2(\Gamma_3)}
\le \alpha \, \| \gamma^*\|
\, ({\bar{c}}_0 + {\bar{c}}_1 \| \gamma \| \| v \|_{V_0} )
\ \ \mbox{for all} \ \ v \in V_0,
\end{equation*}
where $\|\gamma\|$ denotes the norm of the trace operator.
In order to establish pseudomonotonicity of the operator
$B$, we take into account~\cite[Proposition~3.58(ii)]{MOS},
and prove that $B$ is generalized pseudomonotone.

Let $v_n$, $v \in V_0$, $v_n \to v$ weakly in $V_0$, $v^*_n$, $v^* \in V_0^*$, $v^*_n \to v^*$ weakly in $V_0^*$,
$v_n^* \in Bv_n$ and\
$\limsup\, \langle v_n^*, v_n - v \rangle \le 0$.
We show that
$$
v^* \in Bv \ \ \ \mbox{and} \ \ \
\langle v_n^*, v_n \rangle \to \langle v^*, v \rangle.
$$
From condition $v_n^* \in Bv_n$, it follows
$v_n^* = \alpha\, \gamma^* \eta_n$ with
$\eta_n \in \partial J(\gamma v_n)$.
By the estimate (p3), it is clear that $\{ \eta_n \}$ remains in a bounded subset of $L^2(\Gamma_3)$.
Thus, at least for a subsequence, denoted in the same way,
we may suppose that $\eta_n \to \eta$
weakly in $L^2(\Gamma_3)$ with $\eta \in L^2(\Gamma_3)$.
Using the compactness of the trace operator, we have
$\gamma v_n \to \gamma v$ in $L^2(\Gamma_3)$
Now, we employ the strong-weak closedness of the graph of $\partial J$,
see~Proposition~\ref{PROP1}\,(v), to obtain
$\eta \in \partial J(\gamma v)$.
On the other hand,
by $v_n^* = \alpha\, \gamma^* \eta_n$,
it follows $v^* = \alpha\, \gamma^* \eta$.
Hence, we get
$v^* \in \alpha\, \gamma^* \partial J(\gamma v) = Bv$.
Now, it is obvious that
\begin{equation*}
	\langle v_n^*, v_n \rangle =
	\alpha
	\langle \eta_n, \gamma v_n \rangle_{L^2(\Gamma_3)}
	\longrightarrow \alpha
	\langle \eta, \gamma v \rangle_{L^2(\Gamma_3)} =
	\langle \alpha \gamma^* \eta, v \rangle =
	\langle v^*, v \rangle .
\end{equation*}

\noindent
This completes the proof that $B$ is generalized pseudomonotone. Hence, the operator~$B$ is also pseudomonotone.

Subsequently, we note that $A \colon V_0 \to V_0^*$ is pseudomonotone, see~\cite[Theorem~3.69]{MOS}, since it is linear, bounded and nonnegative. Therefore, $A$ is
pseudomonotone and bounded as a multivalued operator
from $V_0$ to $2^{V_0^*}$, see~\cite[Section~3.4]{MOS}.
Since the sum of multivalued pseudomonotone operators
remains pseudomonotone, see~\cite[Proposition~3.59\,(ii)]{MOS},
we infer that $A+B$ is bounded and pseudomonotone.

Next, we prove that the operator $A+B$ is coercive.
In view of the coercivity of $A$, it is enough to show that
\begin{equation}\label{STAR}
\langle Bv, v \rangle \ge -d_0 - d_1 \| v\|_{V_0}
\ \ \mbox{for all} \ \ v \in V_0
\end{equation}
with $d_0$, $d_1 \ge 0$.
First, from hypothesis $H(j)$(d),
by Proposition~\ref{PROP1}\,(i)-(ii),
we have
\begin{eqnarray*}
&&
j^0(x, r; -r) = j^0(x, r; b-r-b) \le
j^0(x, r;b-r) + j^0(x, r; -b) \\ [2mm]
&&\quad
\le j^0(x, r; -b) \le |\partial j(x, r)| \, |-b|
\le |b| (c_0 + c_1 |r|)
\end{eqnarray*}
for all $r \in \real$, a.e. $x \in \Gamma_3$.
Next, let $v \in V_0$, $v^* \in Bv$. Thus,
$v^* = \alpha\, \gamma^* \eta$ with
$\eta \in \partial J(\gamma v)$.
Hence, by the definition of the generalized gradient and
the property (p2), we obtain
\begin{eqnarray*}
&&
\alpha \, \langle \eta, -\gamma v \rangle_{L^2(\Gamma_3)}
\le
\alpha \, J^0(\gamma v; - \gamma v) \le \alpha \int_{\Gamma_{3}} j^0(\gamma v; -\gamma v) \, d\Gamma \\
&&\quad
\le \alpha \, |b| \int_{\Gamma_{3}}
(c_0 + c_1 |\gamma v (x)|)\, d\Gamma  \le
d_0 + d_1 \| v \|_{V_0}
\end{eqnarray*}
with $d_0$, $d_1 \ge 0$.
Using the latter and the equality
$$
\alpha \, \langle \eta, \gamma v \rangle_{L^2(\Gamma_3)}
= \langle \alpha \gamma^* \eta, v \rangle
= \langle v^*, v \rangle,
$$
we deduce
\begin{equation*}
\langle v^*, v \rangle \ge -d_0 - d_1 \| v\|_{V_0}
\ \ \mbox{for all} \ \ v \in V_0
\end{equation*}
which proves (\ref{STAR}). In consequence, we have
$$
\langle (A+B)v, v \rangle \ge \| v \|_{V_0}^2
- d_1 \| v\|_{V_0} - d_0.
$$
We conclude that the multivalued operator $A+B$ is bounded,
pseudomonotone, and coercive,
hence surjective, see~\cite[Proposition~3.61]{MOS}.
We infer that there exists $u \in V_0$ such that
$(A+B) u \ni f$.

In the final step of the proof,
we observe that any solution $u\in V_0$
to the inclusion $(A+B) u \ni f$
is a solution to problem (\ref{abstract}).
Indeed, we have
$$
A u + \alpha\, \gamma^* \eta = f \ \ \mbox{with} \ \
\eta \in \partial J(\gamma u)
$$
and hence
$$
\langle Au, v \rangle +
\alpha \langle \eta, \gamma v \rangle_{L^2(\Gamma_{3})}
= \langle f, v \rangle
$$
for all $v \in V_0$.
Combining the latter with the definition of the generalized gradient and the property (p2), we obtain
\begin{eqnarray*}
&&
\langle f, v \rangle
= \langle Au, v \rangle +
\alpha \, \langle \eta, \gamma v \rangle_{L^2(\Gamma_{3})}
\le
\langle Au, v \rangle + \alpha \, J^0(\gamma u; \gamma v)
\\
&&
\quad
\le
a(u, v) + \alpha \int_{\Gamma_{3}} j^0(\gamma u; \gamma v)
\, d\Gamma
\end{eqnarray*}
for all $v \in V_0$.
This means that $u\in V_0$ solves problem (\ref{abstract}).
This completes the proof.
\end{proof}

\section{Comparison results}\label{Comparison}

In this section we study the following two problems under the standing hypothesis $H(j)$ on the superpotential.

For every $\alpha > 0$, we consider
the hemivariational inequality of the form
\begin{equation}\label{P1}
\mbox{find} \ \ u \in V_0 \ \ \mbox{such that} \ \
a(u, v) + \alpha \int_{\Gamma_{3}}j^{0}(u; v)\, d\Gamma
\geq L(v) \ \ \mbox{\rm for all} \ \  v\in V_{0}
\end{equation}
and the weak form of the elliptic equation
\begin{equation}\label{P2}
\mbox{find} \ \ u_{\infty}\in K \ \ \mbox{such that}\ \
a(u_{\infty},v)=L(v)
\ \ \mbox{for all} \ \ v\in K_{0}.
\end{equation}
Recall that
\begin{equation*}
K=\{v\in V \mid
v = 0 \ \ \mbox{on} \ \ \Gamma_{1},\
v = b \ \ \mbox{on} \ \ \Gamma_{3} \},
\quad
K_{0}=\{v\in V \mid
v = 0 \ \ \mbox{on} \ \ \Gamma_{1}\cup \Gamma_3 \},
\end{equation*}

It follows from Theorem~\ref{existence} that
for each $\alpha > 0$, problem (\ref{P1}) has a solution
$u_\alpha \in V_0$
while~\cite[Corollary 2.102]{CLM} entails that problem~(\ref{P2})
has a unique solution $u_\infty \in K$.
Moreover, it is easy to observe that
problem~(\ref{P2}) can be equivalently
formulated as follows
\begin{equation}\label{P3}
\mbox{find} \ \ u_{\infty}\in K \ \ \mbox{such that}\ \
a(u_{\infty},v-u_\infty) =
L(v-u_\infty)
\ \ \mbox{for all} \ \ v\in K.
\end{equation}

In what follows we need the hypothesis on the data.

\medskip

\noindent
${\underline{(H_0)}}$: \quad
$g \in L^2(\Omega)$, $g \le 0$ in $\Omega$,
$q \in L^2(\Gamma_2)$, $q \ge 0$ on $\Gamma_2$.

%


\begin{Theorem}\label{Theorem5}
If $H(j)$, $(H_0)$ hold and $b\ge 0$, then
\begin{itemize}
\item [\rm (a)] $u_{\alpha}\leq b$ \ {\rm in} \ $\Omega$,
\item [\rm (b)] $u_{\alpha}\leq u_{\infty}$
\ {\rm in} \ $\Omega$,
\end{itemize}
where $u_\alpha \in V_0$ is a solution to problem~\eqref{P1} and $u_\infty \in K$ is the unique solution to problem~\eqref{P2}.
\end{Theorem}
\begin{proof}
(a) \ 	
Let $w = u_\alpha-b$. We shall prove that $w^+ = 0$,
where $r^+ = \max\{ 0, r\}$ for $r \in \real$.
Since $w\big|_{\Gamma _{1}}=-b \le 0$, we have
$w^+\big|_{\Gamma _{1}}=0$.
We choose $v = -w^+ \in V_0$ in problem~(\ref{P1})
to get
\begin{equation*}
a(u_\alpha, -w^+) + \alpha \int_{\Gamma_{3}}
j^0(u_\alpha; -w^+) \, d\Gamma \ge L(-w^+).
\end{equation*}
By the linearity of the form $a$, we easily obtain
\begin{equation*}
a(u_\alpha, -w^+) = - a(w^+, w^+),
\end{equation*}
while $(H_0)$ implies $L(w^+) \le 0$.
Hence
\begin{equation*}
-a(w^+, w^+) + \alpha \int_{\Gamma_{3}}
j^0(u_\alpha; -w^+) \, d\Gamma \ge L(-w^+) \ge 0,
\end{equation*}
and
\begin{equation*}
a(w^+, w^+) \le \alpha \int_{\Gamma_{3}}
j^0(u_\alpha; -(u_\alpha-b)^+) \, d\Gamma.
\end{equation*}
Subsequently, $H(j)$(d) entails
\begin{equation}\label{H1prime}
j^0(x, r; -(r-b)^+) \le 0 \ \ \mbox{for all}
\ \ r \in \real, \ \mbox{a.e.} \ x\in\Gamma_{3}.
\end{equation}
Indeed, if $r \le b$, then $(r-b)^+ =0$ and
$j^0(x, r; -(r-b)^+) = j^0(x, r; 0) = 0 \le 0$.
If $r > b$, we would have $(r-b)^+ = r-b$ and
$j^0(x, r; -(r-b)^+) = j^0(x, r; b-r) \le 0$.
Using the coercivity condition (\ref{coercive}) of the form $a$ and (\ref{H1prime}), we deduce
$m_a \| w^+ \|_V^2 \le 0$.
Hence $w^+  = 0$ in $\Omega$, and finally $u_\alpha \le b$ in $\Omega$.

\smallskip

(b) \
We denote $w = u_\alpha -u_\infty$. It is enough to
show that $w^+ = 0$ in $\Omega$.
We observe that $w \big|_{\Gamma_{1}}=0$.
This allows to choose $v = -w^+ \in V_0$
in problem (\ref{P1}) to obtain
\begin{equation*}
a(u_\alpha-u_\infty, -w^+) + a(u_\infty, -w^+)
+ \alpha \int_{\Gamma_{3}}
j^0(u_\alpha; -w^+) \, d\Gamma \ge L(-w^+).
\end{equation*}
Exploiting the relation
$a(u_\alpha - u_\infty, -w^+) = -a(w^+, w^+)$,
we have
\begin{equation}\label{N1}
-a(w^+, w^+) + a(u_\infty, -w^+)
+ \alpha \int_{\Gamma_{3}}
j^0(u_\alpha; -w^+) \, d\Gamma \ge L(-w^+).
\end{equation}
%
Next, part (a) of the proof shows that
$$
w \big|_{\Gamma_{3}}
= (u_\alpha - b) \big|_{\Gamma_{3}} \le 0
$$
and $w^+ \big|_{\Gamma_{3}}=0$, and consequently
$w^+ \in K_0$.
Since $u_\infty \in K$ solves (\ref{P2}),
taking $v = w^+ \in K_0$ in equality (\ref{P2}),
and using the result in (\ref{N1}), it follows that
\begin{equation*}
a(w^+, w^+) \le \alpha \int_{\Gamma_{3}}
j^0(u_\alpha; -w^+) \, d\Gamma.
\end{equation*}
Since $u_\infty = b$ on $\Gamma_{3}$, by
(\ref{H1prime}), we get
\begin{equation*}
j^0(x, u_\alpha; -(u_\alpha-u_\infty)^+) =
j^0(x, u_\alpha; -(u_\alpha-b)^+) \le 0
\ \ \mbox{a.e. on} \ \Gamma_{3}.
\end{equation*}
Again, by the coercivity of the form $a$,
we have
$m_a \| w^+ \|_V^2 \le 0$.
Therefore, $w^+  = 0$ in $\Omega$, and finally $u_\alpha \le u_\infty$ in $\Omega$. This completes the proof.
\end{proof}

Note that properties (a) and (b) of Theorem~\ref{Theorem5}
obtained for the hemivariational inequality~\eqref{P1}
have been provided for linear elliptic problem (\ref{Pvariacional}) in properties (ii) and (iii) of Theorem~\ref{teor1}.

In what follows, we comment on the monotonicity property analogous to condition
(iv) stated for problem (\ref{Pvariacional}) in Theorem~\ref{teor1}.

\begin{Proposition}\label{PPP}
Assume that $H(j)$ and $(H_0)$ hold, and
\begin{equation}\label{HHH}
j^0(x, r; -(r-s)^+) + c \, j^0(x,s; (r-s)^+) \le 0
\end{equation}
for all $c \ge 1$,
all $r$, $s \in \real$,
a.e. $x\in\Gamma_3$.
Let $u_{\alpha_i} \in V_0$
denote the unique solution to the inequality \eqref{P1}
corresponding to $\alpha_i > 0$,
$i=1$, $2$.
Then the following monotonicity property holds:
\begin{equation*}
\alpha_1 \le \alpha_2
\ \ \Longrightarrow \ \
u_{\alpha_1} \le u_{\alpha_2}
\ \ \mbox{in} \ \ \Omega.
\end{equation*}
\end{Proposition}
\begin{proof}
Let $0 < \alpha_1 \le \alpha_2$ and
$w = u_{\alpha_1} - u_{\alpha_2}$ in $\Omega$.
It is sufficient to prove that $w^+ = 0$ in $\Omega$.
Since $w \big|_{\Gamma_{1}} = 0$, we have $w^+ \in V_0$.
We
choose $v = -w^+ \in V_0$ in problem \eqref{P1}
for $\alpha_1$, and
$v = w^+ \in V_0$ in problem \eqref{P1}
for $\alpha_2$ to get
\begin{eqnarray*}
&&
a(u_{\alpha_1}, -w^+) + \alpha_1 \int_{\Gamma_{3}}
j^{0}(u_{\alpha_1}; -w^+)\, d\Gamma \geq L(-w^+), \\
&&
a(u_{\alpha_2}, w^+) + \alpha_2 \int_{\Gamma_{3}}
j^{0}(u_{\alpha_2}; w^+)\, d\Gamma \geq L(w^+).
\end{eqnarray*}
By adding the last two inequalities, we have
\begin{equation*}
-a(w, w^+) + \alpha_1 \int_{\Gamma_{3}}
j^{0}(u_{\alpha_1}; -w^+)\, d\Gamma
+ \alpha_2 \int_{\Gamma_{3}}
j^{0}(u_{\alpha_2}; w^+)\, d\Gamma \geq 0
\end{equation*}
which implies
\begin{eqnarray*}
&&
a(w^+, w^+) \le \int_{\Gamma_{3}}
\Big( \alpha_1 \, j^{0}(u_{\alpha_1}; -w^+)
+ \alpha_2 \, j^{0}(u_{\alpha_2}; w^+) \Big) \, d\Gamma \\
&&\quad
=
\alpha_1 \int_{\Gamma_{3}}
\Big( j^{0}(u_{\alpha_1}; -w^+)
+ \frac{\alpha_2}{\alpha_1} \, j^{0}(u_{\alpha_2}; w^+) \Big)
\, d\Gamma \le 0.
\end{eqnarray*}
Using the coercivity of the form $a$, we deduce that $w^+ = 0$,
which completes the proof.
\end{proof}

Note that hypothesis (\ref{HHH}) implies that the function
$j(x, \cdot)$ is convex for a.e. $x \in \Gamma_3$.
In fact, if $r>s$, then $(r-s)^+=r-s$ and (\ref{HHH}) gives
$$
j^0(x, r; s-r) + c \, j^0(x,s; r-s) \le 0 \ \ \mbox{for \ all}\ \ c\geq 1.
$$
In particular, taking $c=1$ we obtain the condition equivalent to the relaxed monotonicity condition 
with $m_j=0$, which means that $j(x,\cdot)$ is convex (see Remark~\ref{RRCC}).
  
We conclude that the monotonicity property of Proposition~\ref{PPP} holds for convex potentials,
i.e., for variational inequalities.
The proof of the monotonicity property for hemivariational inequalities remains an open problem.

\section{Asymptotic behavior of solutions}\label{Asymptotic}

In this section we investigate the asymptotic behavior of solutions to problem~(\ref{P1}) when $\alpha \rightarrow\infty$.
To this end, we need the following additional hypothesis
on the superpotential~$j$.

\medskip

\noindent
${\underline{(H_1)}}$: \quad
if $j^0(x, r; b-r) = 0$ for all $r \in \real$,
a.e. $x \in \Gamma_{3}$, then $r = b$.

\begin{Theorem}\label{Theorem6}
Assume $H(j)$, $(H_0)$ and
$(H_1)$.
Let $\{ u_\alpha \} \subset V_0$
be a sequence of solutions to problem~\eqref{P1}
and $u_\infty \in K$ be the unique solution to problem~\eqref{P2}.
Then $u_\alpha \to u_\infty$ in $V$, as $\alpha \to \infty$.
\end{Theorem}

\begin{proof}
First, we prove the estimate on the sequence
$\{ u_\alpha \}$ in $V$.
We choose $v = u_\infty - u_\alpha \in V_0$
as a test function in problem~(\ref{P1}) to obtain
\begin{equation*}
a(u_\alpha, u_\infty -u_\alpha)
+ \alpha \int_{\Gamma_{3}}
j^0(u_\alpha; u_\infty -u_\alpha) \, d\Gamma
\ge L(u_\infty -u_\alpha).
\end{equation*}
From the equality
$a(u_\alpha, u_\infty -u_\alpha) =
- a(u_\infty -u_\alpha, u_\infty -u_\alpha)
+ a(u_\infty, u_\infty -u_\alpha)$,
we get
\begin{equation}\label{N2}
a(v,v) - \alpha \int_{\Gamma_{3}}
j^0(u_\alpha; v) \, d\Gamma \le a(u_\infty, v) - L(v).
\end{equation}
We observe that
$j^0(x, u_\alpha; v) = j^0(x, u_\alpha; b-u_\alpha)$
on $\Gamma_{3}$, and by $H(j)$(d), we have
$j^0(x, u_\alpha; v) \le 0$ on $\Gamma_{3}$.
Hence
\begin{equation*}
a(v,v) \le a(u_\infty, v) - L(v).
\end{equation*}
By the boundedness and coercivity of $a$,
we infer
\begin{equation*}
m_a \| v \|_V^2 \le (M \|u_\infty\|_V
+ \| L \|_{V^*}) \, \| v \|_V
\end{equation*}
with $M > 0$,
and subsequently
\begin{equation}\label{ZZZ}
\| u_\alpha \|_V \le \| v \|_V + \| u_\infty \|_V
\le \frac{1}{m_a}
(M \|u_\infty\|_V  + \| L \|_{V^*}) +
\| u_\infty \|_V =: C,
\end{equation}
where $C > 0$ is independent of $\alpha$.
Hence, since $a(v,v) \ge 0$, from (\ref{N2}),
we have
\begin{equation*}
- \alpha \int_{\Gamma_{3}}
j^0(u_\alpha; v) \, d\Gamma \le
(M \|u_\infty\|_V  + \| L \|_{V^*}) \, \| v \|_V
\le 
\frac{1}{m_a} (M \|u_\infty\|_V  + \| L \|_{V^*})^2
=: C_1,
\end{equation*}
where $C_1 > 0$ is independent of $\alpha$.
Thus
\begin{equation}\label{N3}
- \int_{\Gamma_{3}}
j^0(u_\alpha; v) \, d\Gamma \le
\frac{C_1}{\alpha}.
\end{equation}
It follows from (\ref{ZZZ}) that
$\{ u_\alpha \}$ remains in a bounded subset
of $V$. Thus,
there exists $u^* \in V$ such that,
by passing to a subsequence if necessary, we have
\begin{equation}\label{CONV8}
u_\alpha \to u^* \ \ \mbox{weakly in} \ \ V, \ \mbox{as} \
\alpha \to \infty.
\end{equation}

Next, we will show that $u^* = u_\infty$.
We observe that $u^* \in V_0$ because
$\{ u_\alpha \} \subset V_0$ and $V_0$ is sequentially weakly closed in $V$.
Let $w \in K$ and $v = w - u_\alpha \in V_0$.
From (\ref{P1}), we have
\begin{equation*}
L(w-u_\alpha) \le a(u_\alpha, w -u_\alpha)
+ \alpha \int_{\Gamma_{3}}
j^0(u_\alpha; w-u_\alpha) \, d\Gamma.
\end{equation*}
Since $w = b$ on $\Gamma_3$, by $H(j)$(d), we have
\begin{equation*}
\alpha \int_{\Gamma_{3}} j^0(u_\alpha; w-u_\alpha) \, d\Gamma
= \alpha \int_{\Gamma_{3}} j^0(u_\alpha; b-u_\alpha) \, d\Gamma \le 0
\end{equation*}
which implies
\begin{equation}\label{INEQ2}
L(w-u_\alpha) \le a(u_\alpha, w -u_\alpha).
\end{equation}
Next, we use the weak lower semicontinuity of the functional
$V \ni v \mapsto a(v,v) \in \real$ and from (\ref{INEQ2}),
we deduce
\begin{equation}\label{INEQ4}
u^* \in V_0 \ \ \mbox{satisfies} \ \
L(w-u^*) \le a(u^*, w-u^*) \ \ \mbox{for all} \ \ w \in K.
\end{equation}

Subsequently, we will show that $u^* \in K$.
In fact, from (\ref{CONV8}), by the compactness of the trace operator, we have
$u_\alpha \big|_{\Gamma_{3}} \to
u^* \big|_{\Gamma_{3}}$ in $L^2(\Gamma_3)$, as
$\alpha \to \infty$. Passing to a subsequence if necessary,
we may suppose that
$u_\alpha (x) \to u^*(x)$ for a.e. $x \in \Gamma_3$
and there exists $h \in L^2(\Gamma_3)$ such that
$|u_\alpha(x)| \le h(x)$ a.e. $x \in \Gamma_3$.
Using the upper semicontinuity of the function
$\real\times \real \ni (r, s) \mapsto j^0(x, r; s)
\in \real$ for a.e. $x \in \Gamma_3$,
see~Proposition~\ref{PROP1}(iii),
we get
$$
\limsup j^0(x, u_\alpha(x); u_\infty(x) - u_\alpha(x))
\le j^0(x,u^*(x); u_\infty(x)-u^*(x))
\ \ \mbox{a.e.} \ \ x \in \Gamma_3.
$$
Next, taking into account the estimate
\begin{equation*}
|j^0(x, u_\alpha(x); u_\infty(x) - u_\alpha(x))|
\le (c_0 + c_1 |u_\alpha (x)|) \, |b-u_\alpha(x)| \le k(x)
\ \ \mbox{a.e.} \ \ x \in \Gamma_3
\end{equation*}
with $k \in L^1(\Gamma_3)$ given by
$k(x) = (c_0 + c_1 h(x)) (|b| + h(x))$,
by the dominated convergence theorem, see~\cite[Theorem~2.2.33]{DMP1}, we
obtain
$$
\limsup \int_{\Gamma_3} j^0(u_\alpha; u_\infty - u_\alpha)
\, d\Gamma
\le \int_{\Gamma_3} j^0(u^*; u_\infty-u^*)\, d\Gamma.
$$
Consequently, from $H(j)$(d) and (\ref{N3}), we have
\begin{equation*}
0 \le -\int_{\Gamma_{3}} j^0(u^*; b-u^*) \, d\Gamma
\le \liminf \left(
-\int_{\Gamma_{3}} j^0(u_\alpha; u_\infty-u_\alpha)
\, d\Gamma \right) \le 0
\end{equation*}
which gives
$\int_{\Gamma_{3}} j^0(u^*; b-u^*) \, d\Gamma =0$.
Again by $H(j)$(d), we get $j^0(x, u^*; b-u^*) = 0$
a.e. $x\in\Gamma_{3}$.
Using $(H_1)$, we have $u^*(x) = b$ for a.e.
$x \in \Gamma_3$, which together with (\ref{INEQ4}) implies
\begin{equation*}\label{INEQ5}
u^* \in K \ \ \mbox{satisfies} \ \
L(w-u^*) \le a(u^*, w-u^*) \ \ \mbox{for all} \ \ w \in K.
\end{equation*}

Next, we prove that $u^* = u_\infty$.
To this end, let $v := w -u^* \in K_0$ with
arbitrary $w \in K$.
Hence, $L(v) \le a(u^*, v)$ for all $v \in K_0$.
Recalling that $v \in K_0$ implies $-v \in K_0$,
we obtain
$a(u^*, v) \le L(v)$ for all $v \in K_0$.
Hence, we conclude that
$$
u^* \in K \ \ \mbox{satisfies} \ \
a(u^*, v) = L(v)\ \ \mbox{for all} \ \ v \in K_0,
$$
i.e.,
$u^* \in K$ is a solution to problem (\ref{P2}).
By the uniqueness of solution to problem (\ref{P2}),
we have $u^* = u_\infty$ and hence
$u_\alpha \to u_\infty$ weakly in $V$,
as $\alpha \to \infty$.
From the uniqueness of solution to (\ref{P2}),
we also infer that the whole sequence $\{ u_\alpha \}$ converges weakly in $V$ to $u_\infty$.

Finally, we prove the strong convergence
$u_\alpha \to u_\infty$ in $V$, as $\alpha \to \infty$.
Choosing
$v = u_\infty-u_\alpha \in V_0$ in problem (\ref{P1}),
we obtain
\begin{equation*}
a(u_\alpha, u_\infty -u_\alpha)
+ \alpha \int_{\Gamma_{3}}
j^0(u_\alpha; u_\infty-u_\alpha) \, d\Gamma
\ge L(u_\infty-u_\alpha).
\end{equation*}
Hence
\begin{equation*}
a(u_\infty -u_\alpha, u_\infty -u_\alpha)
\le a(u_\infty, u_\infty -u_\alpha)
+ L(u_\alpha - u_\infty)
+ \alpha \int_{\Gamma_{3}}
j^0(u_\alpha; u_\infty-u_\alpha) \, d\Gamma.
\end{equation*}
Since $u_\infty = b$ on $\Gamma_3$,
by $H(j)$(d) and the coercivity of the form $a$,
we have
$$
m_a \, \| u_\infty -u_\alpha \|^2_V \le
a(u_\infty, u_\infty -u_\alpha)
+ L(u_\alpha - u_\infty).
$$
Employing the weak continuity of both $a(u_\infty, \cdot)$
and $L$, we conclude that
$u_\alpha \to u_\infty$ in $V$, as $\alpha \to \infty$.
This completes the proof.
\end{proof}

\section{Continuous dependence result}\label{Continuous}

In this section we provide the result on continuous dependence
of solution to problem (\ref{P1}) on the internal energy
$g$ and the heat flux $q$ for fixed $\alpha >0$.

First, from the compactness
of the embedding $V$ into $L^2(\Omega)$ and
of the trace operator from $V$ into $L^2(\Gamma)$,
we obtain the following convergence result.
\begin{Lemma}\label{Lemma1}
Let $g_n \in L^2(\Omega)$, $q_n \in L^2(\Gamma_2)$
for $n \in \nat$. Define $L_n \in V^*$, $n \in \nat$,
by
$$
L_n(v) = \int_{\Omega} g_n v \, dx
- \int_{\Gamma_{2}} q_n v \, d\Gamma
\ \ \mbox{\rm for} \ \ v \in V.
$$
If $g_n \to g$ weakly in $L^2(\Omega)$,
$q_n \to q$ weakly in $L^2(\Gamma_{2})$, and
$v_n \in V$, $v_n \to v$ weakly in $V$, then
$$L_n(v_n) \to L(v), \ \ \mbox{\rm as} \ \ n \to \infty,
$$
and there exists a constant $C> 0$ independent of $n$
such that $\| L_n \|_{V^*} \le C$ for all $n \in \nat$.
\end{Lemma}


The continuous dependence result reads as follows.
\begin{Theorem}
Assume that $\alpha > 0$ is fixed, $L$, $L_n \in V^*$,
$n \in \nat$ and $H(j)$ holds.
Let $u_n \in V_0$, $n \in \nat$,
be a solution to problem \eqref{P1} corresponding to $L_n$,
and
\begin{equation}\label{LLL}
\lim L_n(z_n) = L(z) \ \ \mbox{\rm for any} \ \ z_n \to z
\ \mbox{\rm weakly in} \ V, \ \mbox{\rm as} \ n \to \infty.
\end{equation}
Then, there exists a subsequence of $\{ u_n \}$ which converges weakly in $V$ to a solution of problem \eqref{P1} corresponding to $L$.
If, in addition, the following hypotheses hold
\begin{eqnarray}
&&
\hspace{-1.0cm}
j^0(x, r; s-r) + j^0(x, s; r-s) \le m_j \, |r-s|^2
\ \ \mbox{\rm for all} \ \ r, s \in \real,
\ \mbox{\rm a.e.} \ x \in \Gamma_3,
\label{relaxed} \\ [2mm]
&&
\hspace{-1.0cm}
m_a > \alpha \, m_j \| \gamma\|^2, \label{small}
\end{eqnarray}
where $m_j \ge 0$,
then problem \eqref{P1} has the unique solution
$u$ and $u_n \in V_0$
corresponding to $L$ and $L_n$, respectively, and
the whole sequence $\{ u_n \}$ converges to $u$ in $V$,
as $n \to \infty$.
\end{Theorem}
\begin{proof}
Let $u_n \in V_0$ be a solution to problem (\ref{P1})
corresponding to $L_n$, and
$u_\infty \in K$ be the solution to problem (\ref{P2}).
We have
\begin{equation*}
a(u_n, u_\infty -u_n)
+ \alpha \int_{\Gamma_{3}}
j^0(u_n; u_\infty-u_n) \, d\Gamma \ge L_n(u_\infty-u_n).
\end{equation*}
Hence
\begin{equation*}
a(u_\infty - u_n, u_\infty -u_n)
\le a(u_\infty, u_\infty-u_n) + L_n (u_n-u_\infty)
+ \alpha \int_{\Gamma_{3}}
j^0(u_n; b-u_n) \, d\Gamma.
\end{equation*}
From hypothesis $H(j)$(d),
since the form $a$ is bounded and coercive, we get
\begin{eqnarray*}
&&
m_a \| u_\infty - u_n \|_V^2
\le a(u_\infty, u_\infty-u_n) + L_n (u_n-u_\infty) \\ [2mm]
&&
\quad \le
M \| u_\infty \|_V \| u_\infty - u_n \|_V
+ \| L_n \|_{V^*} \| u_\infty - u_n \|_V,
\end{eqnarray*}
and subsequently
\begin{eqnarray*}
\| u_n \|_V \le \| u_\infty - u_n \|_V + \| u_\infty \|_V
\le
\frac{1}{m_a} (M \| u_\infty \|_V + k_1)
+ \| u_\infty \|_V \le k_2
\end{eqnarray*}
for all $n \in \nat$ with $k_1$, $k_2 > 0$ independent
of $n$.
Hence, $\{ u_n \}$ is uniformly bounded in $V$
and also in $V_0$.
From the reflexivity of $V_0$, there exist $\xi \in V_0$
and a subequence of $\{ u_n \}$, denoted in the same way,
such that
\begin{equation*}
u_n \to \xi \ \ \mbox{weakly in} \ V_0,
\ \mbox{as} \ n \to \infty.
\end{equation*}

We will show that $\xi \in V_0$ satisfies (\ref{P1}).
We know that $u_n \in V_0$ and
\begin{equation*}
a(u_n, v) + \alpha \int_{\Gamma_{3}}
j^0(u_n; v) \, d\Gamma \ge L_n(v)
\ \ \mbox{\rm for all} \ \ v \in V_0.
\end{equation*}
Taking the upper limit, we use the weak continuity of $a(\cdot, v)$ and (\ref{LLL}) to get
\begin{equation}\label{CCC}
a(\xi, v)
+ \alpha \limsup \int_{\Gamma_{3}}
j^0(u_n; v) \, d\Gamma \ge \lim L_n(v) = L(v)
\ \ \mbox{\rm for all} \ \ v \in V_0.
\end{equation}
By the compactness of the trace operator from $V$ into $L^2(\Gamma_{3})$, we have
$u_n \big|_{\Gamma_{3}} \to \xi \big|_{\Gamma_{3}}$ in $L^2(\Gamma_3)$, as $n \to \infty$,
and at least for a subsequence,
$u_n(x) \to \xi(x)$ for a.e. $x \in \Gamma_3$
and $|u_n(x)| \le \eta(x)$ a.e. $x \in \Gamma_3$,
where $\eta \in L^2(\Gamma_3)$.
Since the function
$\real\times \real \ni (r, s) \mapsto j^0(x, r; s) \in \real$
a.e. on $\Gamma_3$ is upper semicontinuous, see Proposition~\ref{PROP1}(iii),
we obtain
$$
\limsup j^0(x, u_n(x); v(x)) \le j^0(x, \xi(x); v(x))
\ \ \mbox{a.e.} \ \ x \in \Gamma_3.
$$
Recalling the estimate
\begin{equation*}
|j^0(x, u_n(x); v(x))|
\le (c_0 + c_1 |u_n(x)|) \, |v(x)| \le k(x)
\ \ \mbox{a.e.} \ \ x \in \Gamma_3
\end{equation*}
where $k \in L^1(\Gamma_3)$,
$k(x) = (c_0 + c_1 \eta(x)) |v(x)|$, we apply the dominated convergence theorem, see~\cite[Theorem~2.2.33]{DMP1} to get
$$
\limsup \int_{\Gamma_3} j^0(u_n; v) \, d\Gamma
\le \int_{\Gamma_3} \limsup j^0(u_n; v) \, d\Gamma
\le \int_{\Gamma_3} j^0(\xi; v)\, d\Gamma.
$$
Using the latter in (\ref{CCC}) entails
\begin{equation}\label{DDD}
a(\xi, v)
+ \alpha \int_{\Gamma_{3}}
j^0(\xi; v) \, d\Gamma \ge L(v)
\ \ \mbox{\rm for all} \ \ v \in V_0,
\end{equation}
which means that $\xi \in V_0$ is a solution to problem (\ref{P1}), and completes the first part of the proof.

Next, in addition, we assume (\ref{relaxed}) and (\ref{small}).
The existence of solution to (\ref{P1}) follows from the first part of the theorem.
To prove uniqueness, let $u_1$, $u_2 \in V_0$ solve (\ref{P1}). Then taking as test functions
$u_2-u_1\in V_0$ for $u_1$ and $u_1-u_2\in V_0$ for $u_2$, and adding corresponding inequalities, we obtain
\begin{equation*}
a(u_1-u_2, u_2-u_1) + \alpha \int_{\Gamma_{3}} \left( j^0(u_1; u_2-u_1) + j^0(u_2; u_1-u_2) \right)	
\, d\Gamma \ge 0.
\end{equation*}
From the coercivity of the form $a$ and (\ref{relaxed}),
we have
\begin{equation*}
m_a \, \| u_1-u_2 \|^2_V
\le \alpha \, m_j \int_{\Gamma_{3}} | u_1(x) - u_2(x) |^2 \, d\Gamma
\le \alpha \, m_j \| \gamma \|^2 \, \| u_1-u_2 \|^2_V.
\end{equation*}
Hence, $(m_a - \alpha \, m_j \| \gamma \|^2) \, \| u_1-u_2 \|^2_V \le 0$,
and by the smallness condition (\ref{small}), we get $u_1=u_2$.

Hence, we deduce that solutions $u$, $u_n \in V_0$
to (\ref{P1}) are unique, and by (\ref{DDD}),
we immediately have $\xi = u$.

Finally, we will show the strong convergence of $\{u_n \}$
to $u$ in $V$. We choose suitable test functions from $V_0$ in (\ref{P1}) and (\ref{DDD}) to obtain
\begin{eqnarray*}
	&&
	a(u_n, u-u_n) + \alpha \int_{\Gamma_{3}}
	j^0(u_n; u-u_n) \, d\Gamma \ge L_n(u-u_n) \\
	&&
	a(u, u_n-u) + \alpha \int_{\Gamma_{3}}
	j^0(u; u_n-u) \, d\Gamma \ge L(u_n-u).
\end{eqnarray*}
Adding the two inequalities, we have
\begin{equation*}
a(u_n-u, u-u_n) + \alpha \int_{\Gamma_{3}}
\left( j^0(u_n; u-u_n) + j^0(u; u_n-u) \right)	
\, d\Gamma \ge L_n(u-u_n) + L(u_n-u).
\end{equation*}
Using the coercivity of the form $a$ and (\ref{relaxed}),
we get
\begin{equation*}
m_a \, \| u_n-u \|^2_V
\le \alpha \, m_j \| \gamma \|^2 \, \| u_n-u \|^2_V
+ L_n(u_n-u) + L(u-u_n)
\end{equation*}
which entails
$$
(m_a - \alpha \, m_j \| \gamma \|^2) \, \| u_n-u \|^2_V
\le L_n(u_n-u) + L(u-u_n).
$$
From hypotheses (\ref{LLL}) and (\ref{small}), we deduce
that $\| u_n - u\|_V \to 0$, as $n\to \infty$.
Since $u \in V_0$ is unique, we infer that the whole
sequence $\{u_n \}$ converges weakly in $V$ to $u$.
This proof is complete.
\end{proof}

\begin{Remark}\label{RRCC}
{\rm
It is known that for a locally Lipschitz function $j \colon \real \to \real$,
the condition (\ref{relaxed}) is equivalent to the so-called relaxed monotonicity condition of the subdifferential
\begin{equation}\label{RRR}
(\eta_1 - \eta_2) (r_1-r_2)
\ge - m_j \, | r_1 -r_2 |^2
\end{equation}
for all $r_i \in \real$, $\eta_i \in \partial j(r_i)$,
$i=1$, $2$.
The latter was extensively used in the literature, see~\cite{MOS} and the references therein.
Condition (\ref{relaxed}) can be verified by
proving that the function
$$
\real \ni r \mapsto j(r) + \frac{m_j}{2} |r|^2 \in \real
$$
is nondecreasing.
An example of a nonconvex function which satisfies the condition (\ref{relaxed}) is given in Example~\ref{EXA1}.
Note that if
$j \colon \real \to \real$ is convex, then
(\ref{relaxed}) and (\ref{RRR}) hold with $m_j =0$.
In fact, by convexity,
$$
j^0(r; s-r) \le j(s) - j(r)\quad {\rm and} \quad
j^0(s; r-s) \le j(r) - j(s)
$$
for all $r$, $s \in \real$ which imply
$j^0(r; s-r) + j^0(s; r-s) \le 0$.
Therefore, for a convex function $j \colon \real \to \real$,
condition $(\ref{relaxed})$ or, equivalently, (\ref{RRR}) reduces to monotonicity of the (convex) subdifferential, i.e., $m_j = 0$.
}
\end{Remark}

\section{Examples}\label{Examples}

The following examples provide nonconvex and convex functions which satisfies the hypotheses $H(j)$, $(H_1)$ and (\ref{relaxed}).
\begin{Example}\label{EXA1}
{\rm
Let $j \colon \real \to \real$ be the function defined by
\begin{equation*}
j(r) =
\begin{cases}
(r-b)^2 &\text{{\rm if} \ \  $r < b,$} \\
1-e^{-(r-b)} &\text{{\rm if} \ \ $r\ge b$}
\end{cases}
\end{equation*}
for $r \in \real$ with a constant $b \in \real$.
This function is nonconvex, locally Lipschitz and its subdifferential
is given by
\begin{equation*}
\partial j(r) =
\begin{cases}
2(r-b) &\text{{\rm if} \ \  $r < b,$} \\
[0,1] &\text{{\rm if} \ \ $r = b,$} \\
e^{-(r-b)} &\text{{\rm if} \ \ $r > b$}
\end{cases}
\end{equation*}
for all $r \in \real$.
Hence, we have
$|\partial j(r)| \le 1+ 2|b| + 2|r|$ for all $r \in \real$.
Moreover, using Proposition~\ref{PROP1}(ii), one has
\begin{equation*}
j^0(r; b-r)
=\max\{ \zeta \, (b-r) \mid \zeta \in \partial j(r) \}
=
\begin{cases}
-2(b-r)^2 &\text{{\rm if} \ \  $r < b,$} \\
0 &\text{{\rm if} \ \ $r = b,$} \\
e^{-(r-b)}(b-r) &\text{{\rm if} \ \ $r > b$}
\end{cases}
\end{equation*}
for all $r \in \real$. Thus $H(j)$ is satisfied.
By the above formula, we also infer that $(H_1)$
is satisfied.
Further, we show that condition (\ref{relaxed}) holds
with $m_j = 1$.
The condition (\ref{relaxed}) is equivalent to the relaxed monotonicity of the subdifferential
$$
(\partial j(r) - \partial j(s))(r-s) \ge - |r-s|^2
\ \ \mbox{for all} \ \ r, s \in \real.
$$
The latter means that
$$
\left(
\partial \left(j(r)+\frac{1}{2}r^2 \right)
- \partial \left(
(j(s)+\frac{1}{2} s^2 \right) \right) (r-s) \ge 0
\ \ \mbox{for all} \ \ r, s \in \real,
$$
i.e., the subdifferential $\partial \psi$ of the function $\psi \colon \real \to \real$
defined by
$\psi (r) = j(r) + \frac{1}{2} r^2$
is monotone.
Now, the monotonicity of $\partial \psi$ can be verified
using the formula
\begin{equation*}
\partial \psi(r) =
\begin{cases}
3 r - 2 b &\text{{\rm if} \ \  $r < b,$} \\
[b, b +1] &\text{{\rm if} \ \ $r = b,$} \\
e^{-(r-b)} + r &\text{{\rm if} \ \ $r > b$}
\end{cases}
\end{equation*}
for all $r \in \real$. We conclude that $H(j)$, $(H_1)$ and (\ref{relaxed}) are satisfied.

}
\end{Example}

\begin{Example}
{\rm (see~\cite[Example~3]{MO})
Let the function $j \colon \real \to \real$ be given by
$$j(r) = \min \{ j_1(r), j_2(r) \}$$
for $r \in \real$,
where $j_i \colon \real \to \real$ are convex, quadratic and such that $j_i'(b)=0$, $i=1$, $2$.
It is known, see~\cite[Theorem~2.5.1]{C},
that
$$
\partial j(r) \subset {\rm conv} \{ j_1'(r), j_2'(r) \}
\ \ \mbox{for all} \ \ r \in \real,
$$
so, the subgradient of $j$ has at most a linear growth.
Using the monotonicity of the subgradient of convex function, we get
$$
0\leq \left(j_i'(b)-j_i'(r)\right)(b-r) = -j_i'(r)(b-r)
\ \ \mbox{for all} \ \ r \in \real,\ i=1, 2,
$$
and, by Proposition~\ref{PROP1}(iii), we have
\begin{eqnarray*}
&&
j^0(r; b-r)
= \max \{ \zeta (b-r) \mid \zeta \in \partial j(r)\} \\ [2mm]
&&\quad
= \max \{ \left(\lambda j_1'(r) + (1-\lambda) j_2'(r)\right)(b-r) \mid \lambda \in [0, 1]\} \le 0.
\end{eqnarray*}
Hence, we deduce that condition $H(j)$ is satisfied. Similarly, if $j^0(r; b-r)=0$ for all $r\in\real$, then
$\lambda j_1'(r) (b-r) = 0$ and $(1-\lambda) j_2'(r)(b-r)=0$ for all $r \in \real$ with $\lambda \in [0,1]$, which is possible when $r=b$.
So, $j$ satisfies also $(H_1)$. Further, it is easy to observe that in the case when the graphs of functions $j_1$ and $j_2$ have two common points, then the function $j$ is nonconvex.
}
\end{Example}

\begin{Example}
{\rm
Let $j\colon \real \to \real$ be the function defined
by
$$
j(r)=\frac{1}{2}(r-b)^{2}
$$
for $r \in \real$ with $b \in \real$.
Then
$$
j^{0}(r; s)=(r-b)\, s \ \ \mbox{and} \ \
\partial j(r)=r-b
$$
for $r$, $s \in \real$.
Moreover, we have
$j^0(r; b-r) = (r-b)\, (b-r) = - (b-r)^2 \le 0$
for all $r \in \real$.
Also,
for all $r \in \real$, if $j^0(r; b-r) = 0$, then
$(r-b)\, (b-r) = - (b-r)^2 = 0$, which implies $r =b$.
Hence we deduce that $j$ satisfies properties $H(j)$ and
$(H_1)$.
By Remark~\ref{RRCC}, it is clear that $j$ satisfies
(\ref{relaxed}) with $m_j =0$.
}
\end{Example}

\begin{Example}
{\rm
Let $m_{1}$, $m_{2}$, $r_{0}\in \mathbb{R}$ be constants
such that
$m_{1}\leq -r_{0}<0$ and $m_{2}\geq r_{0}>0$. Consider
the function $j\colon \real \to \real$ defined by
\begin{equation*}
j(r) =
\begin{cases}
\frac{r_{0}^{2}}{2}+m_{1}[r-(b-r_{0})]
&\text{{\rm if} \ \  $r< b-r_{0},$} \\[2mm]
\frac{1}{2}(r-b)^2 &\text{{\rm if} \ \  $b-r_{0}\leq r\leq b+r_{0},$} \\[2mm]
\frac{r_{0}^{2}}{2}+m_{2}[r-(b+r_{0})]
&\text{{\rm if} \ \ $r> b+r_{0}$}
\end{cases}
\end{equation*}
for $r \in \real$, $b \in \real$.
The function $j$ is convex, its subdifferential
is given by
\begin{equation*}
\partial j(r) =
\begin{cases}
m_1 &\text{{\rm if} \ \  $r< b-r_{0},$} \\
\left[ m_{1},-r_{0}\right]
&\text{{\rm if} \ \  $r= b-r_{0},$} \\
r-b &\text{{\rm if} \ \ $b-r_{0}< r < b+r_{0},$} \\
\left[r_{0},m_{2}\right] &\text{{\rm if} \ \ $r= b+r_{0},$} \\
m_2 &\text{{\rm if} \ \  $r> b+r_{0}$} \\
\end{cases}
\end{equation*}
for all $r \in \real$, and its generalized
directional derivative has the form
\begin{equation*}
j^0(r; b-r) =
\begin{cases}
m_{1}(b-r) < 0 &\text{{\rm if} \ \  $r< b-r_{0},$} \\
m_{1}r_{0}  < 0 &\text{{\rm if} \ \  $r= b-r_{0},$} \\
-(b-r)^{2} \leq 0 &\text{{\rm if}\ \ $b-r_{0}< r < b+r_{0},$}\\
r_{0}(b-r)  < 0 &\text{{\rm if} \ \ $r= b+r_{0},$} \\
m_{2}(b-r)< 0 &\text{{\rm if} \ \  $r> b+r_{0}$} \\
\end{cases}
\end{equation*}
for all $r$, $s \in \real$.
Hence, we obtain that $j^{0}(r; b-r)\leq 0$ for all
$r \in \real$. Similarly,
if $j^{0}(r,b-r)=0$ for all $r \in \real$, then $r=b$.
We conclude that $j$ satisfies $H(j)$ and $(H_1)$.
Moreover, the function $j$, being convex,  satisfies (\ref{relaxed}) with $m_j =0$, see Remark~\ref{RRCC}.
}
\end{Example}

\begin{Example}
{\rm
We define $j \colon \real \to \real$ by
\begin{equation*}
j(r) = |r-b| =
\begin{cases}
-r+b &\text{{\rm if} \ \  $r \le b,$} \\
r-b &\text{{\rm if} \ \ $r > b$}
\end{cases}
\end{equation*}
for $r \in \real$ with a constant $b \in \real$.
Then, we have
\begin{equation*}
\partial j(r) =
\begin{cases}
-1 &\text{{\rm if} \ \  $r < b,$} \\
[-1, 1] &\text{{\rm if} \ \  $r = b,$} \\
1 &\text{{\rm if} \ \ $r > b$}
\end{cases}
\end{equation*}
for all $r \in \real$, and
\begin{equation*}
j^0(r; b-r) =
\begin{cases}
b-r &\text{{\rm if} \ \  $r > b,$} \\
0 &\text{{\rm if} \ \  $r = b,$} \\
r-b &\text{{\rm if} \ \ $r < b$}
\end{cases}
\end{equation*}
for all $r \in \real$. Thus,
$j^0(r; b-r) \le 0$ for all $r \in \real$.
Also, we observe that if $j^{0}(r; b-r)=0$
for all $r \in \real$, then $r=b$. In consequence,
the properties $H(j)$ and $(H_1)$ are verified.
Further, since $j$ is convex, it satisfies (\ref{relaxed})
with $m_j =0$, see Remark~\ref{RRCC}.
}
\end{Example}


\section{Conclusions}

We have studied the nonlinear elliptic problem with mixed boundary conditions involving a nonmonotone multivalued subdifferential boundary condition on a part of the boundary.
Based on the notion of the Clarke generalized gradient,
the variational form of the problem leads to an elliptic boundary hemivariational inequality. We have provided results
on existence, comparison of solutions and continuous dependence on the data.
Sufficient conditions have been found which guarantee the asymptotic behavior of solution, when the heat transfer coefficient tends to infinity, to a problem with the Dirichlet boundary condition.
Under our hypotheses, the proof of the monotoni\-ci\-ty property of Theorem~\ref{teor1}(iv) for the elliptic hemivariational inequality (\ref{Pj0alfavariacional}) remains an interesting open problem. We have also given some examples of locally Lipschitz (nondifferentiable and nonconvex) functions to which our results can be applied.




\end{document}